\pdfoutput=1
\RequirePackage{ifpdf}
\ifpdf 
\documentclass[pdftex]{sigma}
\else
\documentclass{sigma}
\fi

\numberwithin{equation}{section}

\newtheorem{Theorem}{Theorem}[section]
\newtheorem{Corollary}[Theorem]{Corollary}
\newtheorem{Lemma}[Theorem]{Lemma}
 { \theoremstyle{definition}
\newtheorem{Remark}[Theorem]{Remark} }

\newcommand{\ee}{{\rm e}}
\newcommand{\ii}{{\rm i}}
\newcommand{\dd}{{\rm d}}
\newcommand{\Ai}{\operatorname{Ai}}
\newcommand{\Bi}{\operatorname{Bi}}
\newcommand{\PII}{{\rm P}_{\rm II}}

\newcommand{\PIV}{{\rm P}_{\rm IV}}

\newcommand{\PXXXIV}{{\rm P}_{\rm 34}}

\begin{document}

\allowdisplaybreaks

\newcommand{\arXivNumber}{1804.00563}

\renewcommand{\thefootnote}{}

\renewcommand{\PaperNumber}{107}

\FirstPageHeading

\ShortArticleName{Large $z$ Asymptotics for Special Function Solutions of Painlev\'e II in the Complex Plane}

\ArticleName{Large $\boldsymbol{z}$ Asymptotics for Special Function Solutions\\ of Painlev\'e II in the Complex Plane\footnote{This paper is a~contribution to the Special Issue on Painlev\'e Equations and Applications in Memory of Andrei Kapaev. The full collection is available at \href{https://www.emis.de/journals/SIGMA/Kapaev.html}{https://www.emis.de/journals/SIGMA/Kapaev.html}}}

\Author{Alfredo DEA\~{N}O}

\AuthorNameForHeading{A. Dea\~{n}o}

\Address{School of Mathematics, Statistics and Actuarial Science, University of Kent, UK}
\Email{\href{mailto:A.Deano-Cabrera@kent.ac.uk}{A.Deano-Cabrera@kent.ac.uk}}

\ArticleDates{Received April 17, 2018, in final form September 22, 2018; Published online October 03, 2018}

\Abstract{In this paper we obtain large $z$ asymptotic expansions in the complex plane for the tau function corresponding to special function solutions of the Painlev\'e II differential equation. Using the fact that these tau functions can be written as $n\times n$ Wronskian determinants involving classical Airy functions, we use Heine's formula to rewrite them as $n$-fold integrals, which can be asymptotically approximated using the classical method of steepest descent in the complex plane.}


\Classification{34M55; 34E05; 33C10; 30E10}

\begin{flushright}
\emph{Dedicated to the memory of Andrei A.~Kapaev}
\end{flushright}

\renewcommand{\thefootnote}{\arabic{footnote}}
\setcounter{footnote}{0}

\section{Introduction and motivation}
The six Painlev\'e differential equations, ${\rm P}_{\rm I}$--${\rm P}_{\rm VI}$, have attracted a great deal of attention in the last decades. They feature in a large and increasing number of areas in mathematics, ranging from random matrix theory to integrable systems (continuous as well as discrete), orthogonal polynomials, partial differential equations and combinatorics. We refer the reader to standard references such as \cite{Clarkson,FIKN,Forrester_loggas,GLS}, as well as the digital library of mathematical functions~\cite[Chapter~32]{DLMF}, for more details and the complete list of the Painlev\'e equations.

Generic solutions of ${\rm P}_{\rm I}$--${\rm P}_{\rm VI}$ are sometimes called Painlev\'e transcendents, or nonlinear special functions, and they cannot be expressed in terms of elementary or even classical special functions. However, for specific values of the parameters in the differential equations, it is known that families of rational and special function solutions exist for ${\rm P}_{\rm II}$--${\rm P}_{\rm VI}$; these appear for instance in the theory of semiclassical orthogonal polynomials, see for example~\cite{VA_discrete} and the recent monograph \cite{VA_OPs}, and in random matrix theory, see~\cite{Forrester_loggas,FW_PII,FW_PIII}.

These rational and special function solutions of Painlev\'e equations can be constructed from a given \emph{seed function} $\varphi(z)$, with a suitable initial value of the parameters; successive application of B\"{a}cklund transformations \cite[Section~4]{Clarkson}, \cite[Section~32.7]{DLMF} leads to a sequence of \emph{tau functions}, that as shown by Okamoto \cite{Okamoto} (see also Forrester and Witte~\cite{FW_PII,FW_PIII} or Clarkson \cite{Clarkson}) have the form of $n\times n$ Wronskian determinants
\begin{gather*}
\tau_n(z) =\det\left(\frac{ \textrm{D}^{j+k}}{\textrm{D} z^{j+k}}\varphi(z)\right)_{j,k=0,1,\ldots, n-1}
\end{gather*}
with initial values $\tau_0(z)=1$ and $\tau_1(z)=\varphi(z)$. Here $\textrm{D}$ is a differential operator that depends on the particular Painlev\'e equation that we are considering. The solution of the Painlev\'e equation (and of other associated equations) can then be written directly in terms of these tau functions.

In this paper we are interested in the special function solutions of the second Painlev\'e equation, denoted $\PII$,
\begin{gather}\label{PII}
q''=zq+2q^3+\alpha, \qquad \alpha\in\mathbb{C}.
\end{gather}

These special function solutions of $\PII$ can be written in terms of standard Airy functions; more precisely, the general seed function is given by
\begin{gather}\label{seed_general}
\varphi(z)=C_1\Ai\big({-}2^{-1/3}z\big)+C_2\Bi\big({-}2^{-1/3}z\big),
\end{gather}
where $C_1$ and $C_2$ are constants, and the tau function is the $n\times n$ Wronskian determinant
\begin{gather}\label{taun}
\tau_n(z)=\det\left(\frac{\dd^{j+k}}{\dd z^{j+k}}\varphi(z)\right)_{j,k=0,1,\ldots,n-1}, \qquad n\geq 1,
\end{gather}
with $\tau_0(z):=1$. As shown in \cite[Theorems 4 and 7]{Clarkson_Airy}, the functions
\begin{gather}\label{qpsigma}
p_n(z)=-2\frac{\dd^2}{\dd z^2}\log\tau_n(z), \qquad q_n(z)=\frac{\dd}{\dd z}\log\frac{\tau_{n-1}(z)}{\tau_{n}(z)}, \qquad \sigma_n(z)=\frac{\dd}{\dd z}\log \tau_{n}(z)
\end{gather}
are special function (Airy) solutions of the $\PXXXIV$ equation
\begin{gather}\label{P34}
p_n\frac{\dd^2 p_n}{\dd z^2}=\frac{1}{2}\left(\frac{\dd p_n}{\dd z}\right)^2+2p_n^3-zp_n^2-\frac{\left(\alpha+\tfrac{1}{2}\right)^2}{2},
\end{gather}
the $\PII$ equation \eqref{PII} and the symmetric $\textrm{S}_{\textrm{II}}$ equation
\begin{gather}\label{SII}
\left(\frac{\textrm{d}^2 \sigma_n}{\textrm{d} z^2}\right)^2+4\left(\frac{\textrm{d} \sigma_n}{\textrm{d} z}\right)^3+2\frac{\textrm{d} \sigma_n}{\textrm{d} z}
\left(z\frac{\textrm{d} \sigma_n}{\textrm{d} z}-\sigma_n\right)=\frac{1}{4}\left(\alpha+\tfrac{1}{2}\right)^2,
\end{gather}
respectively, with $\alpha=n-\tfrac{1}{2}$.

The Airy solutions are used in the recent work of Clarkson, Loureiro and Van Assche~\cite{CLVA}, in the asymptotic analysis of the partition function and free energy in the cubic Hermitian random matrix model by Bleher, Dea\~{n}o and Yattselev \cite{BD_2013,BD_2016,BDY_2017}, and in the study of multiple orthogonal polynomials with a cubic potential in the complex plane by Van Assche, Filipuk and Zhang~\cite{VAFZ_MOP}. In random matrix theory, the pure $\Ai$ case of $\PII$ special function solutions arises in the calculation of averages of powers of the characteristic polynomial in the GUE (Gaussian unitary ensemble), see~\cite[Proposition~28]{FW_PII}. The asymptotic behavior of these Airy solutions has been investigated recently by Clarkson in~\cite{Clarkson_Airy}, but the asymptotic results are restricted to the real line, and only even values of $n$ in the oscillatory regime are included. Asymptotic results for the seed case can also be found in~\cite[Chapter~11]{FIKN}, and Its and Kapaev in \cite[Proposition~4.3]{IK_2000} characterize rational and Airy solutions of~$\PII$ as those that do not exhibit asymptotic elliptic behavior in the complex plane.

The aim of this paper is to obtain large $z$ asymptotic approximations for $\tau_n(z)$ corres\-pon\-ding to special function solutions of $\PII$. As a direct consequence of~\eqref{qpsigma}, the analysis of $\tau_n(z)$ leads directly to asymptotics for the Painlev\'e functions $p_n(z)$, $q_n(z)$ and $\sigma_n(z)$. The methodology is related to ideas used for the large $n$ asymptotics for rational solutions (of $\PII$--$\PIV$) by several authors, including Balogh, Bertola and Bothner~\cite{BBB}, Bertola and Bothner~\cite{BB_IMRN}, Bothner, Miller and Sheng~\cite{BMS_PIII}, Buckingham~\cite{B_PIV}, Buckingham and Miller \cite{BM_PIInoncrit,BM_PIIcrit}. Next we summarise the main steps:
\begin{enumerate}\itemsep=0pt
\item Relate the entries of the Wronskian determinant \eqref{taun}, which are Airy functions, with the moments of a suitably chosen weight function $w(t,z)$,
\begin{gather*}
\mu_m(z)=\int_{\Gamma} t^m w(t,z)\dd t, \qquad m=0,1,2,\ldots,
\end{gather*}
where $\Gamma$ is a suitable contour on the real line or in the complex plane. Namely, if
\begin{gather*}
\frac{\dd^m}{\dd z^m} \varphi(z)=\beta^m\int_{\Gamma} t^m w(t,z)\dd t=\beta^m\mu_m(z),
\end{gather*}
for $m\geq 0$ and some constant $\beta$, then we can write
\begin{gather*}
\tau_n(z)=\det\big(\beta^{j+k}\mu_{j+k}(z)\big)_{j,k=0}^{n-1}=\beta^{n(n-1)}\det\left(\mu_{j+k}(z)\right)_{j,k=0}^{n-1}=\beta^{n(n-1)}D_n(z),
\end{gather*}
for $ n \geq 1$, with $D_0(z):=1$.

It is worth mentioning that often there are several possible weight functions and contours, and some choices may be simpler and/or impose restrictions on parameters.
\item Apply the classical Heine's formula \cite[Corollary~2.1.3]{Ismail}, \cite[Section~2.2]{Szego}, that gives a~multiple integral representation for the Hankel determinant obtained before:
\begin{gather*}
D_{n}(z)=\frac{1}{n!}\int_{\Gamma^n}\Delta_n(t)^2 \prod_{k=1}^{n} w(t_k,z)\dd t_k, \qquad \Delta_n(t)=\prod_{1\leq j<k\leq n} (t_k-t_j).
\end{gather*}
\item Apply the (classical) method of steepest descent to this $n$-fold integral, to obtain the leading asymptotic behavior in different sectors of the complex $z$ plane. The details of this classical asymptotic method in one variable can be found in many references, for instance \cite{BH_asymp,Miller_asymp,Olver_asymp,Temme}, and in the multivariate context, we refer the reader to \cite{Bleistein12saddlepoint}, \cite[Chapter~8]{BH_asymp} or \cite{fedoryuk1989asymptotic}. For convenience, we detail the calculation instead of just writing the leading term in the asymptotic expansion given in \cite[equation~(1.61)]{fedoryuk1989asymptotic}; the main technical details will depend on the different sectors where the variable $z$ grows large in $\mathbb{C}$, which will condition the deformation of $\Gamma$ that is needed, as well as the value of the constants in the seed function, that lead to very different asymptotic behaviors.
\end{enumerate}

\begin{Remark}We observe that in principle the leading term in the asymptotic expansion could also be obtained using the Toda equation satisfied by the tau functions $\tau_n(z)$. This is a~general type of identity that relates consecutive tau functions, and that for $\PII$ reads
\begin{gather}\label{Toda}
\frac{\tau_{n+1}(z)\tau_{n-1}(z)}{\tau_n(z)^2}=\frac{\dd^2}{\dd z^2}\log \tau_n(z), \qquad n\geq 0,
\end{gather}
see other examples in \cite{FW_PII,FW_PIII}.

Making a suitable ansatz of the leading term in the large $z$ asymptotic expansion allows to construct a proof by induction, see \cite[Proposition 5.2]{CJ_PIV} for an example in $\PIV$. However, we find that this methodology does not give a very precise estimate of the subleading terms or the order of the error terms, and also it poses problems in the oscillatory regime, where the leading term as $z\to\infty$ is usually the result of combinations of different exponential contributions, that are difficult to keep track of in this Toda equation. For this reason, we prefer to calculate the expansions using steepest descent of the integral arising from Heine's formula. Once the general structure of the asymptotic expansion for $\tau_n(z)$ is proved, \eqref{Toda} may be used to identify the coefficients therein.
\end{Remark}

\section{Main results}

In the study of the asymptotic behavior of the tau function \eqref{taun} we distinguish, much like in the case of classical Airy functions, two regimes: non-oscillatory (exponential) and oscillatory (trigonometric). Furthermore, in the first case it is enough to obtain asymptotics in the sector $|\arg (-z)|<\frac{\pi}{3}$, since we can use the following rotational symmetries of the seed function \eqref{seed_general}:

\begin{Lemma}\label{lemma_sym}
The Airy seed function $\varphi(z)$ given by \eqref{seed_general} satisfies the follo\-wing identities:
\begin{gather*}
\varphi\big(\ee^{\pm \frac{2\pi\ii}{3}}z\big)= \widetilde{C}_{1\pm}\Ai\big({-}2^{-1/3}z\big)+ \widetilde{C}_{2\pm}\Bi\big({-}2^{-1/3}z\big),
\end{gather*}
with new constants
\begin{gather*}
\widetilde{C}_{1\pm}=\frac{C_1}{2}\ee^{\pm\frac{\pi\ii}{3}}+\frac{3C_2}{2}\ee^{\mp\frac{\pi\ii}{6}}, \qquad
\widetilde{C}_{2\pm}=\frac{C_1}{2}\ee^{\mp\frac{\pi\ii}{6}}+\frac{C_2}{2}\ee^{\pm\frac{\pi\ii}{3}}.
\end{gather*}

The tau function then satisfies
\begin{gather*}
\tau_n\big(\ee^{\pm\frac{2\pi\ii}{3}}z\big)=\ee^{\pm\frac{2\pi\ii}{3}n(n-1)}\tau_n(z).
\end{gather*}
\end{Lemma}

\begin{proof}The proof is a straightforward manipulation of standard formulas for Airy functions, in particular
\begin{gather*}
\Bi(z)=\ee^{-\frac{\pi\ii}{6}}\Ai\big(\ee^{-\frac{2\pi\ii}{3}}z\big)+\ee^{\frac{\pi\ii}{6}}\Ai\big(\ee^{\frac{2\pi\ii}{3}}z\big),\\
\Ai\big(\ee^{\mp \frac{2\pi\ii}{3}}z\big)=\tfrac{1}{2}\ee^{\mp\frac{\pi\ii}{3}}\left(\Ai(z)\pm\ii\Bi(z)\right),
\end{gather*}
see \cite[formulas (9.2.10) and (9.2.11)]{DLMF}. The transformation for $\tau_n(z)$ follows directly from the properties of the seed function.
\end{proof}

Having this result, it is enough to study the tau functions in two different sectors of $\mathbb{C}$:
\begin{itemize}\itemsep=0pt
\item the non-oscillatory sector $|\arg(-z)|<\tfrac{\pi}{3}$,
\item the real axis, where oscillatory behavior occurs.
\end{itemize}

Once the asymptotic behavior is determined in these sectors, the results in the rotated ones follows directly by changing the constants suitably, according to Lemma~\ref{lemma_sym}.

Our main result about the asymptotic behavior of Airy-type solutions of $\PII$ is given in the two theorems below. We present the general result together with a number of consequences, and we highlight the case $C_2=0$ (when the seed function contains only $\Ai$ functions), which is particularly important in applications, because it has a distinguished asymptotic behavior, an extended non-oscillatory sector and particular importance in applications coming from ortho\-gonal polynomials and random matrix theory.

\subsection{Non-oscillatory regime}
\begin{Theorem}\label{Thm_nonosc}
For $n\geq 1$, if we define
\begin{gather}\label{Kn}
K_n=\frac{2^{-\frac{3n^2}{4}-\frac{n}{6}}}{\pi^{\frac{n}{2}}},
\end{gather}
then the function $\tau_n(z)$ has the following asymptotic behavior as $|z|\to\infty$:
\begin{enumerate}\itemsep=0pt
\item[$1.$] If $C_2\neq 0$, in the sector $|\arg (-z)|<\frac{\pi}{3}$,
\begin{gather}\label{taun_nonosc_general}
\tau_n(z)=K_n (-z)^{-\frac{n^2}{4}}\sum_{r=0}^n \mathbf{A}_{n,r}(z) \ee^{\frac{\sqrt{2}}{3}(n-2r)(-z)^{\frac{3}{2}}},
\end{gather}
where
\begin{gather*}
\mathbf{A}_{n,r}(z)\sim(-z)^{\frac{3}{2}r(n-r)} \sum_{j=0}^{\infty} \frac{a_{n,r}^{(j)}}{(-z)^{3j/2}}
\end{gather*}
and the leading coefficient is
\begin{gather*}
a_{n,r}^{(0)}=(-1)^{\left \lfloor{r/2}\right \rfloor}2^{(n-r)(\frac{5r}{2}+1)}C_1^r C_2^{n-r}G(r+1)G(n-r+1),
\end{gather*}
in terms of the Barnes $G$ function, $G(n)=\prod\limits_{k=0}^{n-2} k!$, see {\rm \cite[Section~5.17]{DLMF}}.

\item[$2.$] If $C_2=0$, in the sector $|\arg (-z)|<\pi$,
\begin{gather}\label{taun_nonosc_C20}
\tau_n(z)=K_n (-z)^{-\frac{n^2}{4}}\mathbf{A}_{n,n}(z) \ee^{-\frac{\sqrt{2}}{3}n(-z)^{\frac{3}{2}}}.
\end{gather}
\end{enumerate}
\end{Theorem}

From this theorem and the symmetry relations in Lemma~\ref{lemma_sym}, we can draw a number of consequences. Firstly, we can determine asymptotically free of poles regions in the complex plane for the special function solutions of $\PII$:

\begin{Corollary}If $C_2\neq 0$, then the Airy solutions of $\PII$ are tronqu\'ee solutions $($asymptotically free of poles$)$ in the sectors
\begin{gather*}
S_k=\left\{-\frac{\pi}{3}+\frac{2k\pi}{3}<\arg (-z)<\frac{\pi}{3}+\frac{2k\pi}{3},\, k\in\mathbb{Z}\right\}.
\end{gather*}

If $C_2=0$, then the Airy solutions of $\PII$ are tronqu\'ee solutions in the sector
\begin{gather*}
S= \{|\arg (-z)|<\pi \}.
\end{gather*}
\end{Corollary}

This result proves a conjecture by Clarkson \cite[p.~99]{Clarkson_Airy}. Using these asymptotic expansions, we can determine the asymptotic behavior of the Painlev\'e functions in~\eqref{qpsigma}.

\begin{Corollary}\label{Cor_Painleve_nonosc}For $n\geq 1$, the functions $q_n(z)$, $p_n(z)$ and $\sigma_n(z)$ in~\eqref{qpsigma} admit asymptotic expansions of the following form:
\begin{enumerate}\itemsep=0pt
\item[$1.$] If $C_2\neq 0$, as $|z|\to\infty$ with $|\arg (-z)|<\frac{\pi}{3}$,
\begin{gather*}
\sigma_n(z)=-\frac{n(-z)^{1/2}}{\sqrt{2}}-\frac{n^2}{4z}+\frac{\sqrt{2}n(4n^2+1)}{32(-z)^{5/2}}+\mathcal{O}\big((-z)^{-4}\big),\\
p_n(z)=-\frac{n}{\sqrt{2} (-z)^{1/2}}-\frac{n^2}{2z^2}-\frac{5n(4n^2+1)\sqrt{2}}{32(-z)^{7/2}}+\mathcal{O}\big((-z)^{-5}\big),\\
q_n(z)=\frac{(-z)^{1/2}}{\sqrt{2}}+\frac{2n-1}{4z}-\frac{\sqrt{2}(12n^2-12n+5)}{32(-z)^{5/2}}+\mathcal{O}\big((-z)^{-4}\big).
\end{gather*}
\item[$2.$] If $C_2=0$, as $|z|\to\infty$ with $|\arg (-z)|<\pi$,
\begin{gather*}
\sigma_n(z)=\frac{n(-z)^{1/2}}{\sqrt{2}}+\frac{n^2}{4z}-\frac{\sqrt{2}n(4n^2+1)}{32(-z)^{5/2}}+\mathcal{O}\big((-z)^{-4}\big),\\
p_n(z)=\frac{n}{\sqrt{2} (-z)^{1/2}}-\frac{n^2}{2z^2}+\frac{5n(4n^2+1)\sqrt{2}}{32(-z)^{7/2}}+\mathcal{O}\big((-z)^{-5}\big),\\
q_n(z)=-\frac{(-z)^{1/2}}{\sqrt{2}}+\frac{2n-1}{4z}+\frac{\sqrt{2}(12n^2-12n+5)}{32(-z)^{5/2}}+\mathcal{O}\big((-z)^{-4}\big).
\end{gather*}
\end{enumerate}
\end{Corollary}

\subsection{Oscillatory regime}
The behavior of $\tau_n(z)$ in the oscillatory regime is particularly interesting. Figs.~\ref{fig_Airyodd} and~\ref{fig_Airyeven} show the $\tau_n(z)$ functions for different values of~$n$, on the positive real axis, with $C_2=0$. It is apparent that in the even case there is a~leading algebraic term, with oscillations of small amplitude superimposed, whereas in the odd case the leading term is itself oscillatory with increasing amplitude (except when $n=1$). Theorem~\ref{Thm_osc} makes this idea more precise.

\begin{Theorem}[oscillatory regime]\label{Thm_osc}For $n\geq 1$ and $z\in\mathbb{R}^+$, the function $\tau_n(z)$ has the following asymptotic behavior as $z\to\infty$:
\begin{gather*}
\tau_{2s}(z) = K_{2s}z^{-s^2} \left[ \mathbf{B}_{2s,s}(z) +2\sum_{r=0}^{s-1}\left( \mathbf{B}_{2s,r}(z)\cos(\psi_{2s,r}(z))+\mathbf{D}_{2s,r}(z)\sin(\psi_{2s,r}(z))\right)\right],\\
\tau_{2s-1}(z) = 2K_{2s-1} z^{-\frac{(2s-1)^2}{4}} \sum_{r=0}^{s-1}\left( \mathbf{B}_{2s-1,r}(z)\cos(\psi_{2s-1,r}(z)) +\mathbf{D}_{2s-1,r}(z)\sin(\psi_{2s-1,r}(z))\right),
\end{gather*}
where $K_{n}$ is given by \eqref{Kn}, the phase function is
\begin{gather}\label{phin}
\psi_{n,r}(z)=(n-2r)\left(\frac{\sqrt{2}z^{3/2}}{3}+\frac{n\pi}{4}\right)
\end{gather}
and
\begin{gather*}
\mathbf{B}_{n,r}(z) \sim M_{n,r} z^{\tfrac{3}{2}r(n-r)} \sum_{j=0}^{\infty} \frac{b_{n,r}^{(j)}}{z^{\tfrac{3j}{2}}},\qquad
\mathbf{D}_{n,r}(z) \sim M_{n,r} z^{\tfrac{3}{2}r(n-r)} \sum_{j=0}^{\infty} \frac{d_{n,r}^{(j)}}{z^{\tfrac{3j}{2}}},
\end{gather*}
with
\begin{gather}\label{Mnr}
M_{n,r}=(-1)^{r(n-r)} 2^{\frac{5}{2}r(n-r)}G(r+1)G(n-r+1),
\end{gather}
and leading terms
\begin{gather}
b_{n,r}^{(0)} =\sum_{p=0}^{\left \lfloor{n/2}\right\rfloor} (-1)^{p}C_1^{2p} C_2^{n-2p}
\sum_{q=\max(0,2p-r)}^{\min(2p,n-r)} (-1)^q {r \choose 2p-q}{n-r \choose q},\nonumber\\
d_{n,r}^{(0)} =\sum_{p=0}^{\lfloor{n/2}\rfloor -1} (-1)^{p-1}C_1^{2p+1} C_2^{n-2p-1}\sum_{q=\max(0,2p+1-r)}^{\min(2p+1,n-r)} (-1)^q {r \choose 2p+1-q}{n-r \choose q}.\label{bd}
\end{gather}
\end{Theorem}

\begin{center}
\begin{figure}[h!]\centering
{\includegraphics[width=48mm]{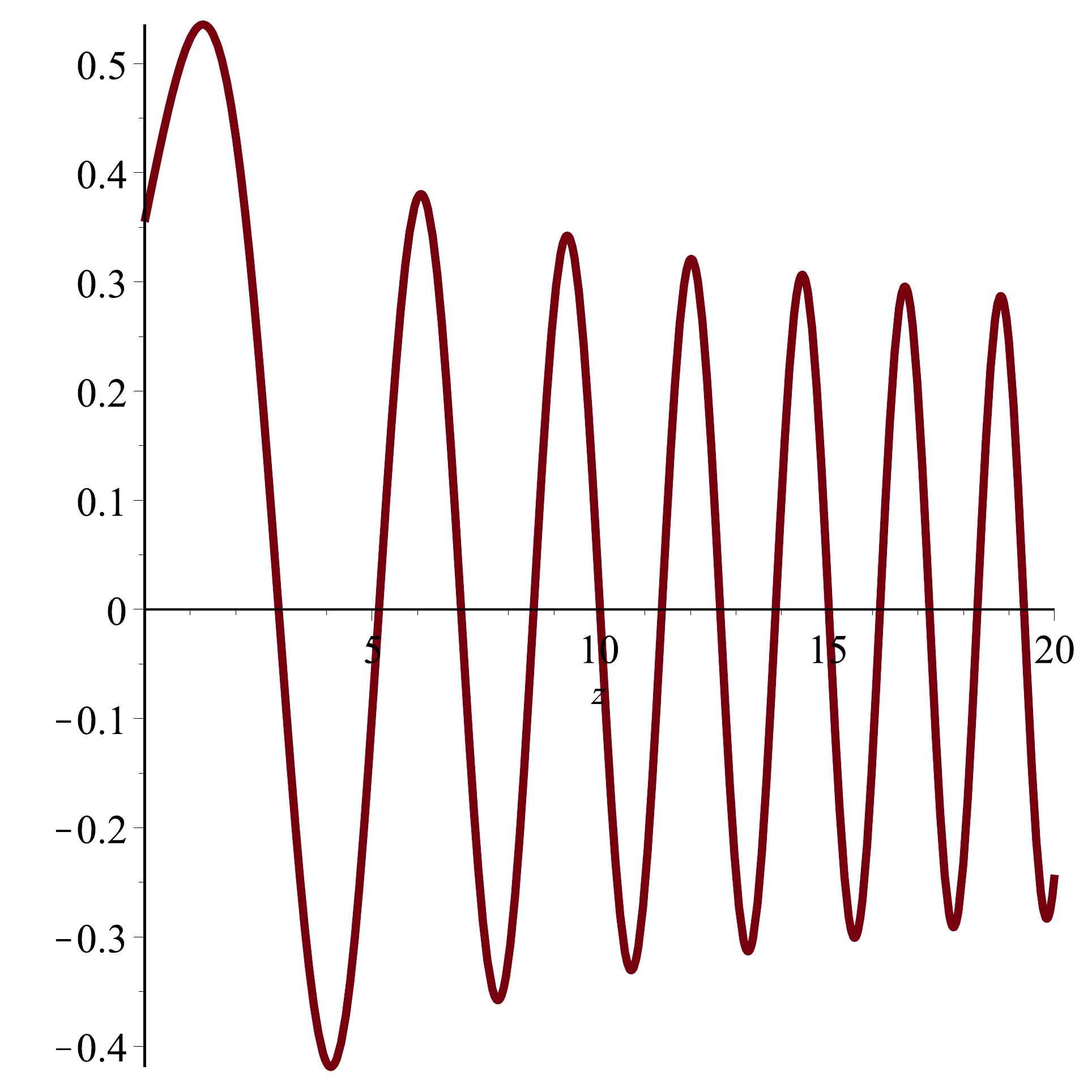}}\hspace*{1mm}
{\includegraphics[width=48mm]{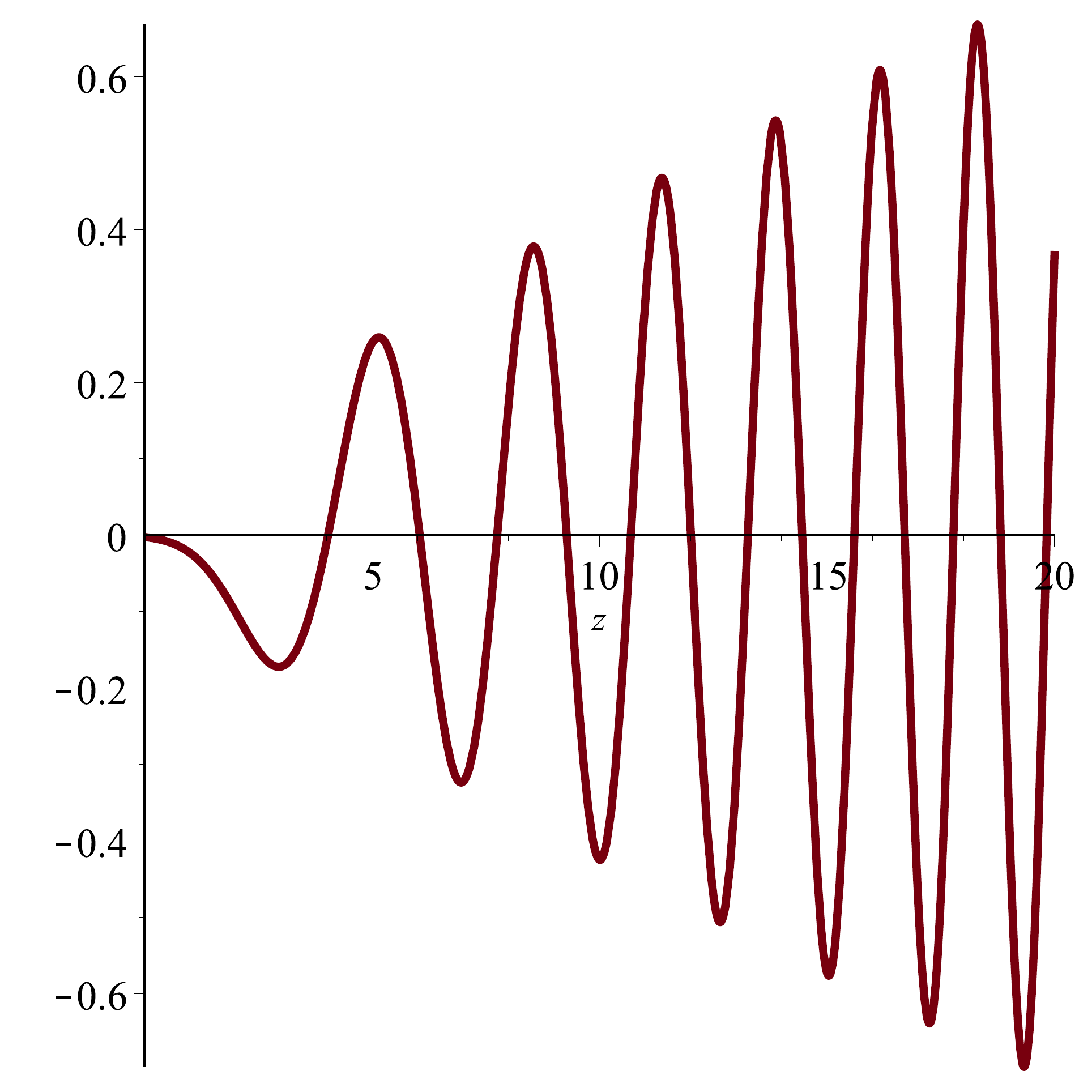}}\hspace*{1mm}
{\includegraphics[width=48mm]{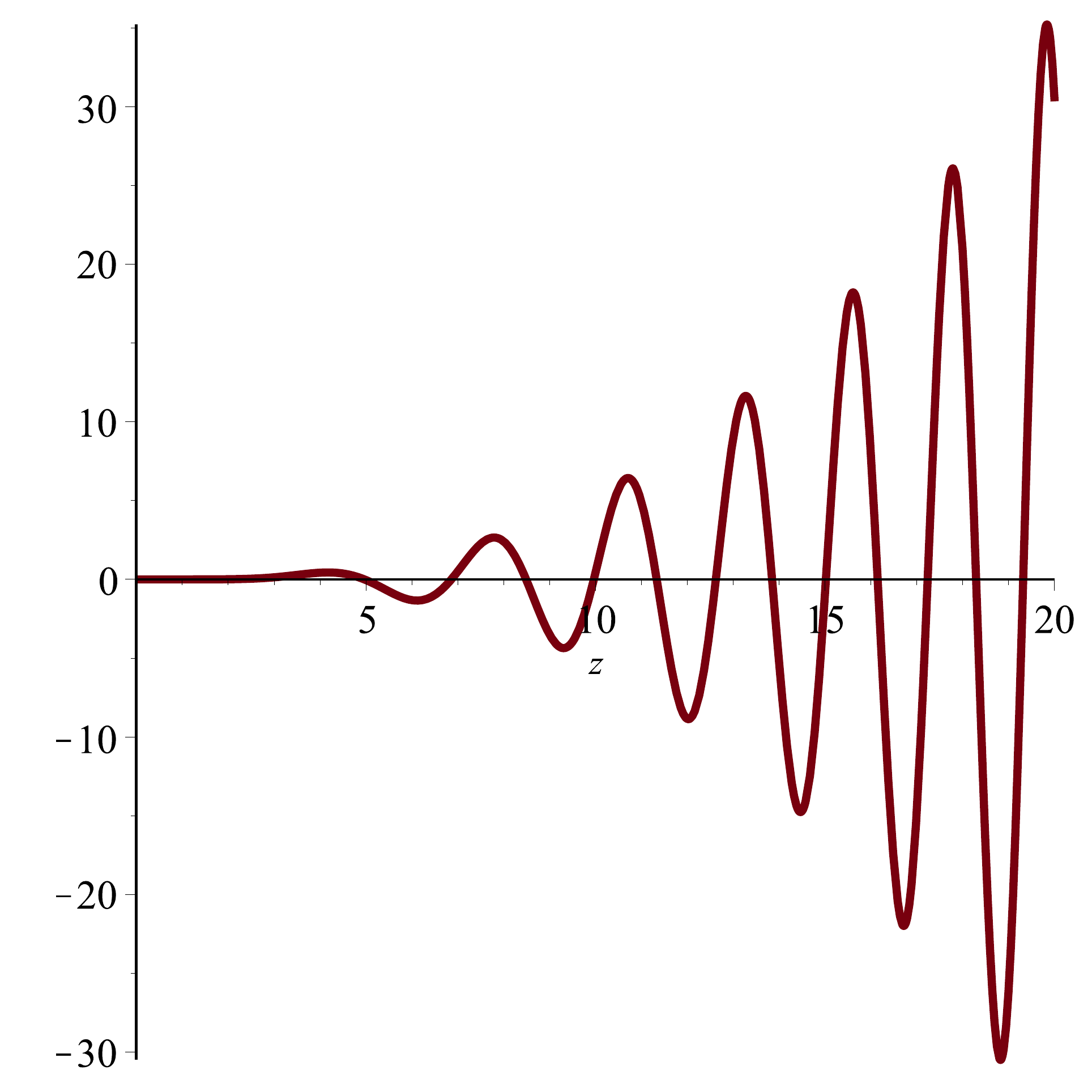}}
\caption{Plots of $\tau_1(z)$ (left), $\tau_3(z)$ (centre), $\tau_5(z)$ (right). In all cases $C_2=0$.} \label{fig_Airyodd}
\end{figure}
\begin{figure}[h!]\centering
{\includegraphics[width=48mm]{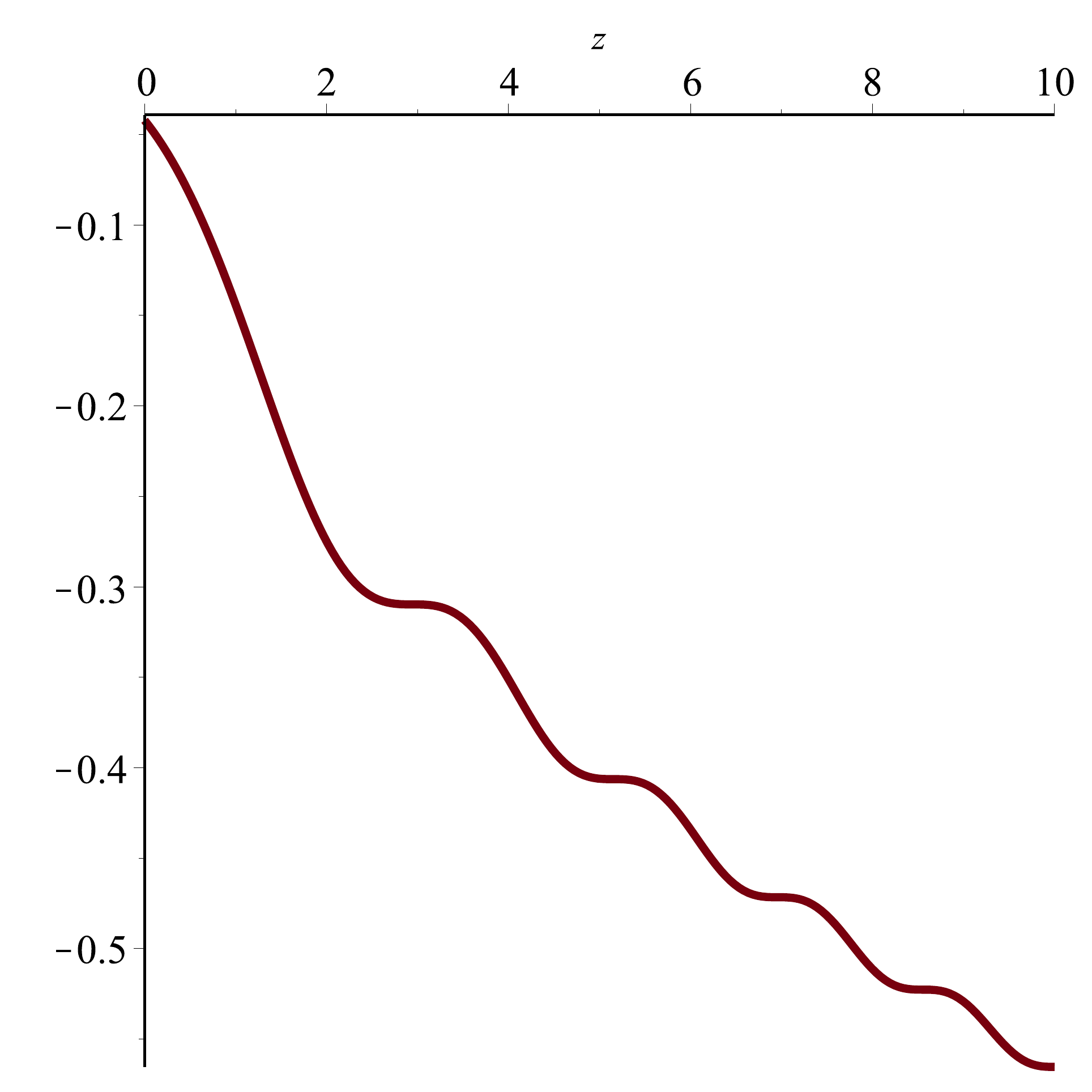}}\hspace*{1mm}
{\includegraphics[width=48mm]{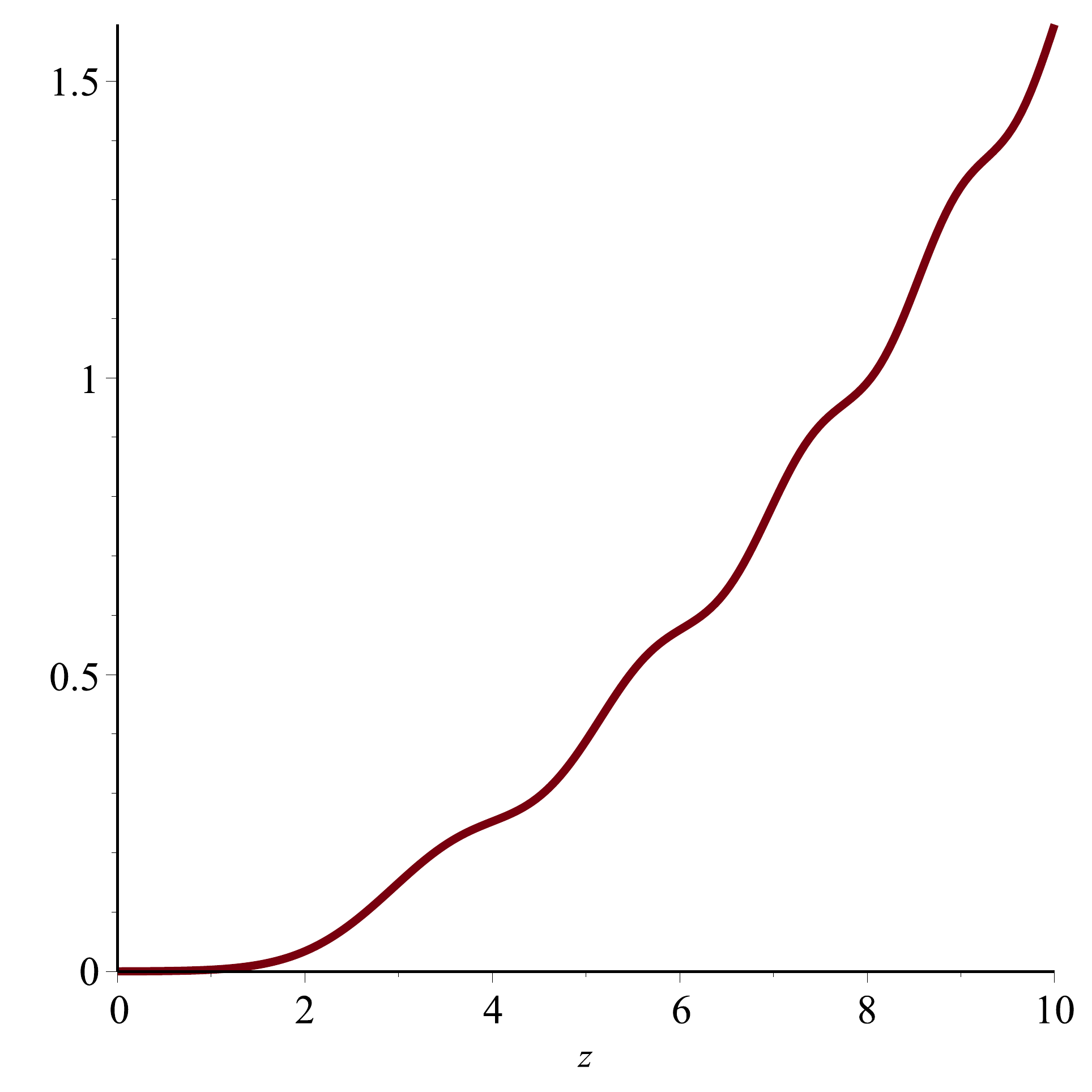}}\hspace*{1mm}
{\includegraphics[width=48mm]{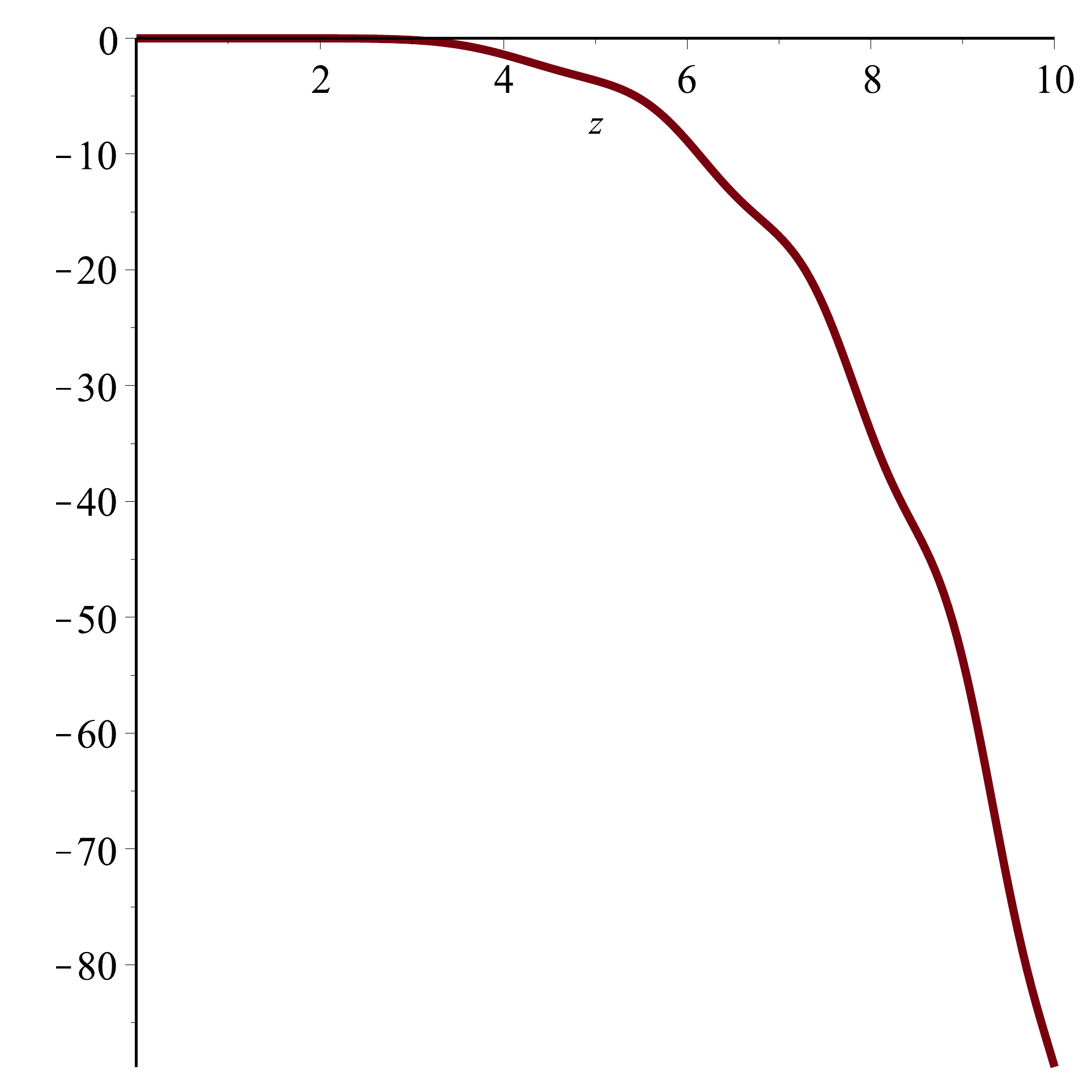}}
\caption{Plots of $\tau_2(z)$ (left), $\tau_4(z)$ (centre), $\tau_6(z)$ (right). In all cases $C_2=0$.}\label{fig_Airyeven}
\end{figure}
\end{center}

\begin{Remark}The general formula for the coefficients $b_{n,r}^{(0)}$ and $d_{n,r}^{(0)}$ is cumbersome, but it can be easily evaluated with symbolic software, and in many cases the sums reduce to just a few terms. Furthermore, several important simplifications can be made if $C_2=0$ (pure $\Ai$ function in the seed): in this case, only the term $p=\lfloor{n/2}\rfloor$ survives, $b_{2s-1,r}^{(0)}=d_{2s,r}^{(0)}=0$, and
\begin{gather*}
b_{2s,r}^{(0)}=(-1)^s C_1^{2s}\sum_{q=\max(0,2s-r)}^{2s-r} (-1)^q {r \choose 2p-q}{n-r \choose q},\\
d_{2s-1,r}^{(0)}=(-1)^s C_1^{2s-1}\sum_{q=\max(0,2s-r-1)}^{2s-r-1} (-1)^q {r \choose 2p+1-q}{n-r \choose q}.
\end{gather*}
\end{Remark}

As before, from this result about the tau function, we can deduce the asymptotic behavior of the Painlev\'e functions in the oscillatory regime.

\begin{Corollary}\label{Cor_Painleve_osc}For $n\geq 1$ and $s\geq 1$, the functions $q_n(z)$, $p_n(z)$ and $\sigma_n(z)$ in~\eqref{qpsigma} admit the following asymptotic expansions as $z\to\infty$, excluding arbitrarily small but fixed neighborhoods of the poles of the leading terms in the approximations:
\begin{enumerate}\itemsep=0pt
\item[$1.$] If $C_2\neq 0$, then
\begin{gather*}
\sigma_{2s}(z) = \frac{s^2}{2z}+\frac{s}{2z} \frac{ b_{2s,s-1}^{(0)}\sin(\psi_{2s,s-1}(z))-d_{2s,s-1}^{(0)}\cos(\psi_{2s,s-1}(z))}
{b_{2s,s}^{(0)}}+\mathcal{O}\big(z^{-5/2}\big),\\
\sigma_{2s-1}(z) =\sqrt{\frac{z}{2}} \frac {-b_{2s-1,s-1}^{(0)}\sin(\psi_{2s-1,s-1}(z))+d_{2s-1,s-1}^{(0)}\cos(\psi_{2s-1,s-1}(z))}
{b_{2s-1,s-1}^{(0)}\cos(\psi_{2s-1,s-1}(z))+d_{2s-1,s-1}^{(0)}\sin(\psi_{2s-1,s-1}(z))}
+\mathcal{O}\big(z^{-1}\big),
\end{gather*}
with coefficients given by \eqref{bd} and phase function \eqref{phin}. Also,
\begin{gather}
p_{2s}(z) =-\frac{\sqrt{2} s}{\sqrt{z}} \frac{b_{2s,s-1}^{(0)}\cos(\psi_{2s,s-1}(z))+d_{2s,s-1}^{(0)}\sin(\psi_{2s,s-1}(z))}{b_{2s,s}^{(0)}}
+\mathcal{O}\big(z^{-2}\big),\nonumber\\
p_{2s-1}(z) =z\left[1+\left(\frac{(b_{2s-1,s-1}^{(0)}\sin(\psi_{2s-1,s-1}(z))-d_{2s-1,s-1}^{(0)}\cos(\psi_{2s-1,s-1}(z))}
{b_{2s-1,s-1}^{(0)}\cos(\psi_{2s-1,s-1}(z))+d_{2s-1,s-1}^{(0)}\sin(\psi_{2s-1,s-1}(z))}\right)^2\right]\nonumber\\
\hphantom{p_{2s-1}(z) =}{} +\mathcal{O}\big(z^{-1/2}\big),\label{p_osc_C2}
\end{gather}
and
\begin{gather}
q_{n}(z) = (-1)^{n} \sqrt{\frac{z}{2}} \frac {-b_{2s-1,s-1}^{(0)}\sin(\psi_{2s-1,s-1}(z))+d_{2s-1,s-1}^{(0)}\cos(\psi_{2s-1,s-1}(z))}
{b_{2s-1,s-1}^{(0)}\cos(\psi_{2s-1,s-1}(z))+d_{2s-1,s-1}^{(0)}\sin(\psi_{2s-1,s-1}(z))}\nonumber\\
\hphantom{q_{n}(z) =}{} +\mathcal{O}\big(z^{-1}\big),\label{q_osc_C2}
\end{gather}
where $s=\left \lceil{\frac{n}{2}}\right\rceil$.

\item[$2.$] If $C_2=0$, then
\begin{gather*}
\sigma_{2s}(z)=\frac{s^2}{2z}-\frac{s}{2z}\sin(\psi_{2s,s-1}(z))+\mathcal{O}\big(z^{-5/2}\big),\nonumber\\
\sigma_{2s-1}(z)=\sqrt{\frac{z}{2}}\cot(\psi_{2s-1,s-1}(z))+\mathcal{O}\big(z^{-1}\big).
\end{gather*}

Also,
\begin{gather}
p_{2s}(z)=\frac{\sqrt{2} s}{\sqrt{z}}\cos(\psi_{2s,s-1}(z))+\mathcal{O}\big(z^{-2}\big),\nonumber\\
p_{2s-1}(z) =\frac{z}{\sin^2(\psi_{2s-1,s-1}(z))}+\mathcal{O}\big(z^{-1/2}\big),\label{p_osc_C20}
\end{gather}
and
\begin{gather}\label{q_osc_C20}
q_{n}(z)=(-1)^{n} \sqrt{\frac{z}{2}}\cot(\psi_{2s-1,s-1}(z))+\mathcal{O}\big(z^{-1}\big), \qquad s=\left \lceil{\frac{n}{2}}\right\rceil.
\end{gather}
\end{enumerate}
\end{Corollary}

We note that this is in accordance with the results in \cite[Theorems~5,~8 and~9]{Clarkson_Airy}, but it also extends the asymptotic results to the complex plane, and it includes the case of $n$ odd in the oscillatory regime. Furthermore, it provides more precise estimates of the remainder terms.

It is worth mentioning that in the reference \cite{IKO_P34}, Kuijlaars, Its and \"{O}stensson consider the asymptotics (for real $z$) of a one parameter family of solutions of $\PXXXIV$, depending on a parameter that is related to $\alpha$ in \eqref{P34}. This family is relevant in the analysis of critical edge behavior in unitary random matrix ensembles \cite{IKO_RMT}, and it includes the Airy solutions as a particular case. The results in \eqref{p_osc_C20} agree with Theorem~1.2 in~\cite{IKO_P34}, for the asymptotic behavior of the tronqu\'ee solutions of the $\PXXXIV$ equation\footnote{In the notation of~\cite{IKO_P34}, we have $b=s_2=(-1)^n$, in terms of the standard Stokes multipliers for $\PII$. Since the restriction $b>0$ applies in the steepest descent analysis of the Riemann--Hilbert problem, we only recover the case of even $n$ for Airy solutions, with $b=1$ and $\beta_0=0$. We also note that the parameter $\alpha$ in~\cite{IKO_P34} corresponds to $-\alpha/2-1/4=-n/2$ in our notation, and a change of variables is needed as well in the solutions of $\PXXXIV$.}.

\begin{Remark}In order to give the leading asymptotic behavior it is enough to keep a few terms in the previous expansion, namely $r=0$ in the non-oscillatory regime, and $r=s$ (if $n=2s$ is even), and $r=s-1$ (if $n=2s-1$ is odd) in the oscillatory regime. We have opted to present the full expansion because in order to examine the asymptotic behavior of solutions of the Painlev\'e equations $\sigma_n(z)$, $p_n(z)$ and $q_n(z)$, in particular in the oscillatory regime, higher order terms are needed. Also, these extra terms give exponential contributions that are important in the Stokes phenomenon for the $\Ai$ solution, we refer the reader to the discussion in Appendix~\ref{ApA}.
\end{Remark}

\section{Proof of Theorem \ref{Thm_nonosc}}\label{proof_nonosc}

We recall the classical integral representations of the Airy functions
\begin{gather}
\Ai(z)=\frac{1}{2\pi\ii}\int_{\infty \ee^{-\pi\ii/3}}^{\infty \ee^{\pi\ii/3}}\exp\big(\tfrac{1}{3}t^{3}-zt\big)\mathrm{d}t,\nonumber\\
\Bi(z) =\frac{1}{2\pi} \left(\int_{-\infty}^{\infty \ee^{-\pi\ii/3}}+\int_{-\infty}^{\infty \ee^{\pi\ii/3}}\right)\exp\big(\tfrac{1}{3}t^{3}-zt\big)\mathrm{d}t,\label{intAiBi}
\end{gather}
for $z\in\mathbb{C}$, see \cite[formulas (9.5.4) and (9.5.5)]{DLMF}. Bearing in mind the form of the seed function~\eqref{seed_general}, we take the weight function
\begin{gather}\label{fullw}
w(t,z)=\exp\big(\tfrac{1}{3}t^{3}+2^{-1/3}zt\big),
\end{gather}
and for $m\geq 0$ we define the moments
\begin{gather*}
\mu^{\Ai}_{m}(z) =\frac{C_1}{2\pi\ii} \int_{\infty \ee^{-\pi\ii/3}}^{\infty \ee^{\pi\ii/3}} t^m w(t,z)\dd t, \\
\mu^{\Bi}_{m}(z) =\frac{C_2}{2\pi} \left(\int_{-\infty}^{\infty \ee^{-\pi\ii/3}}+\int_{-\infty}^{\infty \ee^{\pi\ii/3}}\right) t^m w(t,z)\dd t.
\end{gather*}

Then, we have $\varphi(z)=\mu^{\Ai}_0(z)+\mu^{\Bi}_0(z)$, cf.~\eqref{seed_general}, and as a direct consequence,
\begin{gather*}
\frac{\dd^k}{\dd z^k} \varphi(z)=2^{-\frac{k}{3}}\big[\mu^{\Ai}_k(z)+\mu^{\Bi}_k(z)\big], \qquad k\geq 0.
\end{gather*}

The Wronskian determinant constructed from the seed function and the Hankel determinant for the weight function \eqref{fullw} are therefore related as follows:
\begin{gather*}
\tau_n(z)=\det\left(\frac{\dd^{j+k}}{\dd z^{j+k}}\varphi(z)\right)_{j,k=0}^{n-1}=
2^{-\frac{n(n-1)}{3}}	\det\big(\mu^{\Ai}_{j+k}(z)+\mu^{\Bi}_{j+k}(z)\big)_{j,k=0}^{n-1}.
\end{gather*}

Consider now that we have a total of $r$ Airy $\Ai$ integrals in the determinant, $0\leq r \leq n$, and consequently~$n-r$
Airy $\Bi$ integrals. Then
\begin{gather}\label{taunDnr}
\tau_n(z)= 	\frac{2^{-\frac{n(n-1)}{3}}}{(2\pi)^n} 	\sum_{r=0}^n \frac{C_1^r C_2^{n-r}}{\ii^r} D_{n,r}(z),
\end{gather}
where the Hankel determinants $D_{n,r}(z)$ can be written, following the standard theory \cite[Corollary 2.1.3]{Ismail}, as $n$-fold integrals:
\begin{gather}\label{Dnr}
D_{n,r}(z)=\frac{1}{n!}\int\cdots\int \Delta_n(\mathbf{t})^2 \prod_{k=1}^{n} w(t_k,z)\dd t_k, \qquad \Delta_n(\mathbf{t})=\prod_{1\leq j<k\leq n} (t_k-t_j),
\end{gather}
where $r$ integrals are taken along the path corresponding to $\Ai$ and $n-r$ along the one for $\Bi$, and we denote $\mathbf{t}=(t_1,t_2,\ldots,t_n)$. Note that the integrand is symmetric in the $n$ variables (the Vandermonde determinant may change sign when permuting variables, but it appears squared), therefore we can suppose, without loss of generality, that the first $r$ integrals correspond to $\Ai$ and the last $n-r$ integrals to~$\Bi$, and then sum over the possible permutations.

Using the scaled variables $\mathbf{t}=2^{-1/6}\sqrt{\rho} \mathbf{u}$, where
\begin{gather*}
z=-\rho\ee^{\ii\alpha}, \qquad \rho\geq 0, \qquad |\alpha|<\frac{\pi}{3},
\end{gather*}
cf.~\cite[Section 7.3]{BH_asymp}, we obtain
\begin{gather*}
D_{n,r}(z)=\frac{\left(2^{-1/6} \rho^{1/2}\right)^{n^2}}{n!}{n \choose r}\int_{\Gamma_{\alpha}}\cdots\int_{\Gamma_{\alpha}} \Delta_n(\mathbf{u})^2 \exp\left(\frac{\rho^{3/2}}{\sqrt{2}}\phi(\mathbf{u})\right)\prod_{k=1}^{n} \dd u_k,
\end{gather*}
where the phase function is
\begin{gather*}
\phi(\mathbf{u}) = \sum_{k=1}^n \left[\frac{1}{3}u_k^{3}-\ee^{\ii\alpha}u_k\right],
\end{gather*}
and $\Gamma_{\alpha}$ is any smooth infinite path that joins the sectors $\infty \ee^{-\pi\ii/3}$ and $\infty\ee^{\pi\ii/3}$ (for the $\Ai$ case) and the sectors $-\infty $ and $\infty\ee^{\pm\pi\ii/3}$ (for the $\Bi$ case).

Clearly, the gradient and Hessian of this function are
\begin{gather*}
\nabla \phi(\mathbf{u}) = \big(u_1^2-\ee^{\ii\alpha}, \ldots, u_n^2-\ee^{\ii\alpha}\big), \qquad
H \phi(\mathbf{u}) = 2\operatorname{diag}(u_k)_{k=1,\ldots,n}.
\end{gather*}

The stationary points
\begin{gather*}
\mathbf{u}^* = (u_{\pm},\ldots,u_{\pm}), \qquad u_{\pm}=\pm\ee^{\ii\alpha/2}.
\end{gather*}
(with any combination of signs) are non-degenerate, since $\det H \phi(\mathbf{u}^*)\neq 0$.

For $|\alpha|<\frac{\pi}{3}$, the main contribution to each $\Bi$ integral is given by the stationary point $u_{k-}$, since $\textrm{Re}(\phi(u_{k-}))>0$ and $\textrm{Re}(\phi(u_{k+}))<0$, and it will appear doubled because of the two paths joining $-\infty$ and $\infty \ee^{\pm\pi\ii/3}$ in \eqref{intAiBi}; for the $\Ai$ integrals, the relevant stationary point is $u_{k+}$, since path deformation through $u_{k-}$ would change the asymptotic behavior (from exponentially decreasing to exponentially increasing). Therefore, for the asymptotic analysis we need to consider stationary points of the form
\begin{gather}\label{statpoint}
\mathbf{u}^{(r)}= (\underbrace{u_+,\ldots,u_+}_{r\, \text{times}},\underbrace{u_-,\ldots,u_-}_{n-r\, \text{times}}).
\end{gather}

We add and subtract the value of the phase function at this point, to obtain
\begin{gather*}
D_{n,r}(z) = \frac{\left(2^{-1/6} \rho^{1/2}\right)^{n^2}}{r!(n-r)!}
\exp\left(\frac{\rho^{3/2}}{\sqrt{2}}\phi(\mathbf{u}^{(r)})\right)\\
\hphantom{D_{n,r}(z) =}{} \times \int_{\Gamma_{\alpha}}\cdots\int_{\Gamma_{\alpha}} \Delta_n(\mathbf{u})^2
\exp\left(\frac{\rho^{3/2}}{\sqrt{2}}\big[\phi(\mathbf{u})-\phi\big(\mathbf{u}^{(r)}\big)\big]\right)\prod_{k=1}^{n} \dd u_k.
\end{gather*}

By analyticity of the integrand in all the variables~$u_k$, the precise structure of these global paths is not relevant for the analysis, as long as they connect the correct sectors in the complex plane, since we can deform the contours sequentially in the different variables. We will use this freedom to integrate along the paths of steepest descent through the points~$\mathbf{u}^{(r)}$. These paths, denoted here by~$\Gamma_{\alpha}$, are implicitly given by the equation
\begin{gather*}
\operatorname{Im} \phi(\mathbf{u})=\operatorname{Im} \phi\big(\mathbf{u}^{(r)}\big),
\end{gather*}
and then it follows that the function $\phi(\mathbf{u})-\phi(\mathbf{u}^{(r)})$ is real-valued on $\Gamma_{\alpha}$. The paths of steepest descent are difficult to describe globally for general values of~$\alpha$, see the discussion in \cite[Sec\-tion~7.3]{BH_asymp}, however, for asymptotic approximations we only need their existence around the stationary points.

We isolate the stationary points by fixing $\delta>0$ and two discs $D\big(\mathbf{u}^{(r)},\delta\big)$ of radius $\delta$; then we define
\begin{gather*}
\Gamma_{\alpha,\delta}=\Gamma_{\alpha}\cap D\big(\mathbf{u}^{(r)},\delta\big),
\end{gather*}
that is, small portions of the steepest descent path around the stationary points. Then as $z\to\infty$ in the sector that we are considering, we have
\begin{gather}\label{DnIn_exp}
D_{n,r}(z)=\exp\left(\frac{\sqrt{2}}{3}(n-2r)(-z)^{3/2}\right) [I_{n,r}(z)+E_{n,r}(z) ],
\end{gather}
where
\begin{gather*}
I_{n,r}(z) = \frac{\left(2^{-1/6}\rho^{1/2}\right)^{n^2}}{{r!(n-r)!}}
\int_{\Gamma_{\alpha,\delta}} \Delta_n(\mathbf{u})^2 \exp\left(\frac{\rho^{3/2}}{\sqrt{2}}\big[\phi(\mathbf{u})-\phi\big(\mathbf{u}^{(r)}\big)\big]\right)\prod_{k=1}^n \dd u_k,
\end{gather*}
and $E_{n,r}(z)$ is the remainder.

It is important to note that we need to choose $\delta>0$ in such a way that the remainder $E_{n,r}(z)$ is exponentially small with respect to \emph{all} the exponential terms that are present in~\eqref{DnIn_exp}, that is, $E_{n,r}(z)=o\big(\exp\big({-}\frac{\sqrt{2}n}{3}(-z)^{3/2}\big)\big)$ as $|z|\to\infty$. Computing such a $\delta$ explicitly is complicated in general, but it is clear that for large enough $|z|$, such a choice is always possible, given that the phase function is real (and decaying) along the path of steepest descent.

We apply a final change of variables to transform the exponential terms in the integral into Gaussians, which is a particular case of Morse lemma in the literature \cite[Chapter~1, Section~2]{fedoryuk1989asymptotic}: for each $1\leq k\leq n$, we define
\begin{gather*}
\phi(\mathbf{u})-\phi\big(\mathbf{u}^{(r)}\big)=-\sum_{k=1}^n v^2_k.
\end{gather*}

This change of variable can be written by components, and as $v_k\to 0$ we have
\begin{gather*}
u_{k}=u_++\ii\ee^{-\frac{\alpha\ii}{4}}v_k+\frac{\ee^{-\ii\alpha}}{6}v_k^2+\mathcal{O}\big(v_k^3\big),\qquad
u_{k}=u_-+\ee^{-\frac{\alpha\ii}{4}}v_k+\frac{\ee^{-\ii\alpha}}{6}v_k^2+\mathcal{O}\big(v_k^3\big).
\end{gather*}

This maps the contour $\Gamma_{\alpha,\delta}$ onto $[-\varepsilon,\varepsilon]^n$ on the real axis, for some $\varepsilon>0$. Then, with an exponentially small error again, we can extend the integrals to the whole real axis, using a~standard estimate: for $\varepsilon>1$, $C=\frac{\rho^{3/2}}{\sqrt{2}}>0$, and $f$ analytic, even and with at most polynomial growth, we have
\begin{gather*}
\int_{-\infty}^{\infty} f(v)\ee^{-Cv^2}\dd v- \int_{-\varepsilon}^{\varepsilon} f(v)\ee^{-Cv^2}\dd v
=2\int_{\varepsilon}^{\infty} f(v)\ee^{-Cv^2}\dd v \leq 2\int_{\varepsilon}^{\infty} f(v)\ee^{-Cv}\dd v.
\end{gather*}

The last integral can be written in terms of incomplete Gamma functions (see \cite[Sections~8.2 and~8.11]{DLMF} for definitions and asymptotics) and a similar argument can be used in~$n$ variables, taking the Vandermonde determinant as the function~$f$. If we make~$\varepsilon$ large enough, we have a~remainder that is exponentially small with respect to all the terms in~\eqref{DnIn_exp}.

In order to study the Vandermonde determinant, we split it in three parts, separating those terms that combine two $u_+$ or two $u_-$ values and then a final one that mixes positive and negative stationary points:
\begin{gather*}
\Delta_n(\mathbf{u})^2=\prod_{1\leq j<k\leq r} (u_k-u_j)^2\prod_{r+1\leq j<k\leq n} (u_k-u_j)^2 \prod_{1\leq j\leq r,r+1\leq k\leq n} (u_k-u_j)^2.
\end{gather*}

This term, together with the differentials, becomes
\begin{gather*}
\Delta_n(\mathbf{u})^2 \prod_{k=1}^n \dd u_k=2^{2r(n-r)}\ee^{\frac{\pi\ii}{2}r^2-\frac{\alpha\ii}{4}(n^2+6r(n-r))}\\
\hphantom{\Delta_n(\mathbf{u})^2 \prod_{k=1}^n \dd u_k=}{} \times
\prod_{1\leq j<k\leq r} (v_k-v_j)^2 \prod_{r+1\leq j<k\leq n} (v_k-v_j)^2 \Xi(\mathbf{v})\prod_{k=1}^n \dd v_k,
\end{gather*}
where
\begin{gather*}
\Xi(\mathbf{v})=1+\sum_{k=1}^n a_kv_k+\sum_{j,k=1}^n b_{j,k}v_jv_k+\cdots, \qquad |v|\to 0,
\end{gather*}
for some coefficients $a_k$ and $b_{j,k}$ whose exact form is not relevant for the leading term in the asymptotic expansion. This leads to two decoupled Selberg integrals, and writing everything together we obtain
\begin{gather}
I_{n,r}(z)=\big(2^{-1/6}\rho^{1/2}\big)^{n^2}\exp\left(\frac{\sqrt{2}(n-2r)}{3}(-z)^{3/2}-\left(\frac{n^2}{4}+\frac{3r(n-r)}{2}\right)\alpha\ii
+\frac{r^2\pi\ii}{2}\right) \nonumber\\
\hphantom{I_{n,r}(z)=}{} \times 2^{(n-r)(2r+1)} \frac{S_r(z)S_{n-r}(z)}{r!(n-r)!}\big(1+\mathcal{O}\big(\rho^{-3/2}\big)\big),\label{Inr_final}
\end{gather}
where
\begin{gather}\label{Selberg}
S_d=\int_{\mathbb{R}^d} \Delta_d(v)^2 \prod_{k=1}^d \ee^{-C v_k^2}\dd v_k=\frac{(2\pi)^{d/2}}{(2C)^{d^2/2}}\prod_{k=1}^d k!=\frac{(2\pi)^{d/2}}{(2C)^{d^2/2}}G(d+2),
\end{gather}
for $d\geq 1$ and $\operatorname{Re} C>0$, in terms of the Barnes $G$ function, see \cite[Section~5.17]{DLMF}. Identifying $C=\rho^{3/2}/\sqrt{2}$ and $z=-\rho\, \ee^{\ii\alpha}$, we have
\begin{gather*}
I_{n,r}(z) =2^{-\frac{5n^2}{12}+\frac{n}{2}+(n-r)(\frac{5}{2}r+1)}\pi^{\frac{n}{2}}
(-z)^{-\frac{n^2}{4}-\frac{3}{2}r(n-r)} G(r+1)G(n-r+1)\\
\hphantom{I_{n,r}(z) =}{} \times \exp\left(\frac{\sqrt{2}(n-2r)}{3}(-z)^{3/2}+\frac{r^2\pi\ii}{2}\right)
\big(1+\mathcal{O}\big((-z)^{-3/2}\big)\big).
\end{gather*}

Combining the powers of $2$ and $\pi$ with the prefactor in \eqref{taunDnr} and summing over~$r$, we arrive at the leading term in~\eqref{taun_nonosc_general}.

This calculation gives very relevant information about the remainder as well: the order of the error in~\eqref{Inr_final} comes from the fact that any term beyond the leading one in the differentials or in the Vandermonde will produce linear terms in the components~$v_k$ (which integrate to~$0$ against the Gaussian because of symmetry) and then quadratic terms, which, using the formula
\begin{gather*}
\int_{\mathbb{R}} v_k^2 \ee^{-Cv_k^2}\dd v_k=\frac{1}{2C}\int_{\mathbb{R}} \ee^{-Cv_k^2}\dd v_k, \qquad \operatorname{Re} C>0,
\end{gather*}
will contribute to an error of order $\mathcal{O}\big(\rho^{-3/2}\big)$ as $\rho\to\infty$, since in our situation we have $C=\rho^{3/2}/\sqrt{2}$. This is true for higher order terms as well, so each exponential level in \eqref{taun_nonosc_general} contains in fact a~full asymptotic expansion in inverse powers of $(-z)^{3/2}$, which appears in the coefficients~$\mathbf{A}_{n,r}(z)$ in~\eqref{taun_nonosc_general}.

In the case $C_2=0$ we have no $\Bi$ integrals, so we take $r=n$, and all integrals will involve the stationary point $u_+$ only. The calculation is analogous and the leading term follows from this substitution into~\eqref{taun_nonosc_general}, which gives~\eqref{taun_nonosc_C20}; furthermore, following the standard asymptotic theory of the Airy~$\Ai$ function, see for example \cite[Section~4.7]{Miller_asymp}, the expansion holds in the larger sector $|\arg(-z)|<\pi$.

\section{Proof of Theorem \ref{Thm_osc}}

In the oscillatory regime, the asymptotic analysis is similar, but slightly more complicated because both~$\Ai$ and~$\Bi$ integrals need to be evaluated both at $u_+$ and at $u_-$, which for $z>0$ have exponential contributions with the same real part. We will highlight the main differences with respect to the non-oscillatory case.

We start the calculation with \eqref{taunDnr} and \eqref{Dnr} as before. Assuming that $z>0$, we make the change of variables $\mathbf{t}=2^{-1/6}z^{1/2} \mathbf{u}$, and we have
\begin{gather}\label{Dnr_osc}
D_{n,r}(z) = \frac{\left(2^{-1/6} z^{1/2}\right)^{n^2}}{n!}{n \choose r}\int_{\Gamma_{\alpha}}\cdots\int_{\Gamma_{\alpha}} \Delta_n(\mathbf{u})^2 \exp\left(\frac{\rho^{3/2}}{\sqrt{2}}\phi(\mathbf{u})\right)\prod_{k=1}^{n} \dd u_k,
\end{gather}
where the phase function is now
\begin{gather*}
\phi(\mathbf{u}) =\sum_{k=1}^n \left[\frac{1}{3}u_k^{3}+u_k\right],
\end{gather*}
with stationary points $\mathbf{u}^{*}=(\pm\ii, \ldots,\pm\ii)$. By symmetry, we can consider without loss of generality the stationary point $\mathbf{u}^{(r)}$ in \eqref{statpoint} again, and sum in $r$ over the ${n \choose r}$ possible permutations. Adding and subtracting the phase function at \eqref{statpoint}, we obtain
\begin{gather*}
D_{n,r}(z) =\frac{\left(2^{-1/6} z^{1/2}\right)^{n^2}}{n!}{n \choose r}\\
 \hphantom{D_{n,r}(z) =}{} \times \exp\left(\frac{\sqrt{2}\ii}{3}(n-2r)z^{3/2}\right)\int_{\Gamma^n} \Delta_n(\mathbf{u})^2 \exp\left(\frac{z^{3/2}}{\sqrt{2}}\left[\phi(\mathbf{u})-\phi(\mathbf{u}^{(r)})\right]\right)\prod_{k=1}^n \dd u_k.
\end{gather*}

As before, we isolate these points by fixing $\delta>0$ and two discs $D\big(\mathbf{u}^{(r)},\delta\big)$ of radius $\delta$ around the stationary points in each of the variables $u_k$. We take $\Gamma_{\delta}=\Gamma\cap D\big(\mathbf{u}^{r},\delta\big)$, where $\Gamma$ is the corresponding path of steepest descent. Then as $z\to\infty$ we have
\begin{gather*}
D_{n,r}(z)=\exp\left(\frac{\sqrt{2}\ii}{3}(n-2r)z^{3/2}\right) [ I_{n,r}(z)+E_{n,r}(z) ],
\end{gather*}
where
\begin{gather*}
I_{n,r}(z)=\frac{\left(2^{-1/6}z^{1/2}\right)^{n^2}}{r!(n-r)!} \int_{\Gamma_{\delta}} \Delta_n(\mathbf{u})^2 \exp\left(\frac{z^{3/2}}{\sqrt{2}}\big[\phi(\mathbf{u})-\phi\big(\mathbf{u}^{(r)}\big)\big]\right)\prod_{k=1}^n \dd u_k,
\end{gather*}
and $E_{n,r}(z)$ is the remainder.

In the analysis of the Vandermonde determinant, it is convenient to split the different cases depending on which stationary point is considered, using the index $r$, and not in terms of~$\Ai$ and~$\Bi$ functions. Note that in the non-oscillatory case, both ideas are equivalent, since each Airy function requires only one of the stationary points ($u_+$ for~$\Ai$ and $u_-$ for~$\Bi$). In the oscillatory regime, however, we have (independently of the parameter~$r$), $p$ integrals of type~$\Ai$ and $n-p$ integrals of type $\Bi$, with $0\leq p\leq n$, and any integral around the stationary point~$-\ii$ has different orientation depending if we are integrating along the~$\Ai$ or the~$\Bi$ contour: the contour for an integral along an $\Ai$ contour is oriented from right to left, and each one of them adds a~$-1$ factor. In order to quantify this, for any given~$p$, we need to count all possible configurations where we have $q$ integrals of $\Ai$ type in the last $n-r$ cases (where the point~$-\ii$ is taken into account) and therefore $p-q$ integrals of $\Ai$ type in the first~$r$ cases, for any value of~$q$ from~$0$ to~$p$. We define
\begin{gather*}
H_{n,r,p}=\sum_{q=\max(0,p-r)}^{\min(p,n-r)} {r \choose p-q}{n-r \choose q} (-1)^q,
\end{gather*}
where the limits of summation are set so that all binomial numbers are well defined.

Writing together all the contributions and summing over $p$, we have
\begin{gather*}
I_{n,r}(z)=(-4)^{r(n-r)}\exp\left[(n-2r)\ii\left(\frac{\sqrt{2}z^{3/2}}{3}+\frac{n\pi}{4}\right)\right]
\frac{S_{r}(z) S_{n-r}(z)}{r!(n-r)!} \sum_{p=0}^n H_{n,r,p}\\
\hphantom{I_{n,r}(z)=}{} \times \big(2^{-1/6}z^{1/2}\big)^{n^2} \big(1+\mathcal{O}\big(z^{-3/2}\big)\big),
\end{gather*}
in terms of Selberg integrals \eqref{Selberg} again. This leads to
\begin{gather*}
I_{n,r}(z) = 2^{-\frac{5n^2}{12}+\frac{n}{2}}\pi^{\frac{n}{2}}\, z^{-\frac{n^2}{4}-\frac{3}{2}r(n-r)} \exp\left[(n-2r)\ii\left(\frac{\sqrt{2}z^{3/2}}{3}+\frac{n\pi}{4}\right)\right] M_{n,r}
\sum_{p=0}^n H_{n,r,p}\\
\hphantom{I_{n,r}(z) =}{} \times \big(1+\mathcal{O}\big(z^{-3/2}\big)\big),
\end{gather*}
where the coefficient $M_{n,r}$ is given by \eqref{Mnr}. The powers of $2$ and $\pi$ can then be combined and simplified using~\eqref{taunDnr}, which relates $\tau_n(z)$ and $D_n(z)$, and \eqref{Dnr_osc}.

Finally, the asymptotic expansion can be written in terms of sines and cosines, instead of complex exponentials, noting that $H_{n,n-r,p}=(-1)^p H_{n,r,p}$, for $0\leq r\leq n$, and grouping terms depending on the parity of~$p$. The proof of this symmetry relation can be obtained by writing the different cases: if $p\leq n-r$ and $p-r\leq 0$, then
\begin{gather*}
H_{n,r,p}=\sum_{q=0}^p {r\choose p-q}{n-r\choose q}(-1)^q,
\end{gather*}
and
\begin{gather*}
H_{n,n-r,p} = \sum_{q=0}^p {n-r\choose p-q}{r\choose q}(-1)^q = \sum_{q=0}^p {n-r\choose q}{r\choose p-q}(-1)^{p-q}=(-1)^p H_{n,r,p},
\end{gather*}
where we have reversed the order inside the sum (i.e., $q\mapsto p-q$). If $p\leq n-r$ and $p-r\geq 0$, then
\begin{gather*}
H_{n,r,p}=\sum_{q=p-r}^p {r\choose p-q}{n-r\choose q}(-1)^q,
\end{gather*}
and
\begin{gather*}
H_{n,n-r,p} =\sum_{q=0}^r {n-r\choose p-q}{r\choose q}(-1)^q = \sum_{q=0}^r {n-r\choose p-r+q}{r\choose r-q}(-1)^{r-q}\\
\hphantom{H_{n,n-r,p}}{}= \sum_{s=p-r}^p {n-r\choose s}{r\choose p-s}(-1)^{p-s}=(-1)^p H_{n,r,p},
\end{gather*}
reversing the sum in the first step and shifting $q=s-p+r$ in the second one.

The case $p>n-r$ can be proved in a similar way.

\section{Proof of Corollaries~\ref{Cor_Painleve_nonosc} and~\ref{Cor_Painleve_osc}}
In the non-oscillatory regime, we can derive asymptotic expansions for the Painlev\'e functions $\sigma_n(z)$, $p_n(z)$ and $q_n(z)$ in \eqref{qpsigma} quite straightforwardly. We will use the asymptotic expansions for~$\tau_n(z)$ and~\eqref{qpsigma}, observing that differentiation is permitted since~$\tau_n(z)$ is an analytic function of~$z$ inside the relevant sectors, see \cite[Section~1.8]{Olver_asymp}. In the oscillatory regime, one would need to open a sector around the positive real axis; this calculation is similar to the oscillatory case, writing $z=\rho\ee^{\ii\alpha}$, $|\alpha|<\frac{\pi}{3}$ and considering both stationary points.

Suppose first that $C_2\neq 0$. Instead of working directly with the asymptotic expansion \eqref{taun_nonosc_general}, it is simpler to pick the leading term therein, corresponding to $r=0$, i.e.
\begin{gather*}
\tau_n(z)=K_n \mathbf{A}_{n,0}(z)\ee^{\frac{\sqrt{2}}{3}n(-z)^{\frac{3}{2}}}+\mathcal{O}\Big(\ee^{\frac{\sqrt{2}}{3}(n-2)(-z)^{\frac{3}{2}}}\Big)
\end{gather*}
as $|z|\to\infty$ in the sector $|\arg (-z)|<\frac{\pi}{3}$. Using this result, we can deduce the form of the asymptotic expansion for the function $\sigma_n(z)$, which is
\begin{gather}\label{asymp_sigma_nonosc}
\sigma_n(z)=\frac{\tau_n'(z)}{\tau_n(z)}=\sum_{k=0}^{K-1} s_{n,k}(-z)^{\tfrac{1}{2}-\tfrac{3k}{2}}+\mathcal{O}\Big((-z)^{\tfrac{1}{2}-\tfrac{3K}{2}}\Big), \qquad K\geq 1.
\end{gather}

Then, substituting this expansion into the differential equation \eqref{SII}, we can identify the coefficients $s_{n,k}$. Then, we can compute
\begin{gather*}
p_n(z)=-2\sigma'_n(z)=-2 \sum_{k=0}^{K-1} p_{n,k}(-z)^{-\tfrac{1}{2}-\tfrac{3k}{2}}+\mathcal{O}\left((-z)^{\tfrac{1}{2}-\tfrac{3K}{2}}\right),
\end{gather*}
where the coefficients $p_{n,k}$ follow easily from $s_{n,k}$. Finally, from \eqref{qpsigma} again, we have
\begin{gather*}
q_n(z)=\sigma_{n-1}(z)-\sigma_n(z)=\sum_{k=0}^{K-1} q_{n,k}(-z)^{-\tfrac{1}{2}-\tfrac{3k}{2}}+\mathcal{O}\Big((-z)^{\tfrac{1}{2}-\tfrac{3K}{2}}\Big),
\end{gather*}
with coefficients that follow from $s_{n,k}$ once again.

In the case $C_2=0$, we can use a similar argument, but with the form
\begin{gather*}
\tau_n(z)=K_n \mathbf{A}_{n,n}(z)\ee^{-\frac{\sqrt{2}}{3}n(-z)^{\frac{3}{2}}}.
\end{gather*}

The negative exponential term naturally leads to changes in the coefficients, but the form of the asymptotic expansion is the same.

In the oscillatory regime, calculations are more delicate, and the main difficulty is to establish a general pattern for the asymptotic expansion, as we did before in \eqref{asymp_sigma_nonosc}. The reason for this is that in Theorem \ref{Thm_osc} there is a clear leading term, but as we expand further, different trigonometric functions will be involved. These higher order terms can be computed using symbolic software, but for brevity we will give only the leading terms.

The leading terms in Theorem \ref{Thm_osc} clearly correspond to $r=s$ if $n=2s$ is even, and to $r=s-1$ if $n=2s-1$ is odd. For $s\geq 1$, we have
\begin{gather*}
\sigma_{2s}(z)=\frac{\tau_{2s}'(z)}{\tau_{2s}(z)} =
\frac{s^2}{2z}+\frac{s}{2z}\,\frac{ b_{2s,s-1}^{(0)}\sin(\psi_{2s,s-1}(z))-d_{2s,s-1}^{(0)}\cos(\psi_{2s,s-1}(z))}
{b_{2s,s}^{(0)}}+\mathcal{O}\big(z^{-5/2}\big),
\end{gather*}
with coefficients given by \eqref{bd} and phase function \eqref{phin}.

If $C_2=0$, this expression simplifies considerably, since only the term $p=s$ survives, and then $b_{2s,s}^{(0)}=-b_{2s,s-1}^{(0)}=C_1^{2s}$ and $d_{2s,s-1}^{(0)}=0$, so
\begin{gather*}
\sigma_{2s}(z) = \frac{s^2}{2z}-\frac{s}{2z}\sin(\psi_{2s,s-1}(z))+\mathcal{O}\big(z^{-5/2}\big).
\end{gather*}

This is in agreement with \cite[Theorem~9]{Clarkson_Airy}, correcting the phase function. In the odd case, we have
\begin{gather*}
\sigma_{2s-1}(z) =\sqrt{\frac{z}{2}} \frac
{-b_{2s-1,s-1}^{(0)}\sin(\psi_{2s-1,s-1}(z))+d_{2s-1,s-1}^{(0)}\cos(\psi_{2s-1,s-1}(z))}
{b_{2s-1,s-1}^{(0)}\cos(\psi_{2s-1,s-1}(z))+d_{2s-1,s-1}^{(0)}\sin(\psi_{2s-1,s-1}(z))}
+\mathcal{O}\big(z^{-1}\big).
\end{gather*}

If $C_2=0$, this expression simplifies again
\begin{gather*}
\sigma_{2s-1}(z)=\sqrt{\frac{z}{2}}\cot(\psi_{2s-1,s-1}(z))+\mathcal{O}\big(z^{-1}\big).
\end{gather*}

In both cases, the asymptotic approximation is valid away from the zeros of the denominators that appear in the leading terms, we refer the reader to \cite[Chapter 1, Section~8.1]{Olver_asymp} for the general theory.

Similar calculations lead to the asymptotic expansions for $p_n(z)$, although care is needed because the leading terms may come from subleading ones before, as a result of differentiation of the trigonometric functions. Straightforward manipulations, using the fact that $p_n(z)=-2\sigma_n'(z)$ and $q_n(z)=\sigma_{n-1}(z)-\sigma_n(z)$, for $n\geq 1$, which follows from~\eqref{qpsigma}, lead to~\eqref{p_osc_C2}, \eqref{q_osc_C2}, \eqref{p_osc_C20}, and~\eqref{q_osc_C20}.

Higher order terms in these asymptotic expansions can be computed by using the corresponding differential equations. This is quite straightforward in the non-oscillatory regime, but more involved in the oscillatory one, since several trigonometric functions with different phase functions intervene.

\appendix
\section[Stokes phenomenon for the $\Ai$ solution]{Stokes phenomenon for the $\boldsymbol{\Ai}$ solution}\label{ApA}

The case $C_2=0$, where the seed function only contains the Airy~$\Ai$ functions, is especially relevant both because of its asymptotic behavior and in applications. It is interesting to observe that in this case we have strong asymptotics for $\tau_n(z)$ in the whole cut plane $|\arg(-z)|<\pi$ involving only one exponential factor, see \eqref{taun_nonosc_C20}; however, we can also use the rotational symmetry given by Lemma~\ref{lemma_sym}, to study the asymptotics in the regions $0<\arg z<\frac{2\pi}{3}$ and \smash{$-\frac{2\pi}{3}<\arg z<0$}. This calculation uses a different seed function in the original sector $|\arg (-z)|<\frac{\pi}{3}$ and it gives subdominant exponential terms. These do not affect the leading asymptotic behavior, but they give a non-linear Stokes phenomenon for this family of solutions of $\PII$, in the spirit of Its and Kapaev \cite{IK_2003}, see also \cite[Chapter~11]{FIKN}.

For instance, if we consider the sector $0<\arg z<\frac{2\pi}{3}$ and the seed function with $\widetilde{C}_1=1$ and $\widetilde{C}_2=0$
(corresponding to the pure Airy $\Ai$ function in this rotated sector), we have from \eqref{taun_nonosc_C20} the following asymptotic approximation
\begin{gather*}
\tau_n(z)=K_n (-z)^{-\frac{n^2}{4}}\mathbf{A}_{n,n}(z)\ee^{-\frac{\sqrt{2}}{3}n (-z)^{\frac{3}{2}}}\\
\hphantom{\tau_n(z)}{} =(-1)^{\left \lfloor{n/2}\right \rfloor} K_n G(n+1)(-z)^{-\frac{n^2}{4}}
\big(1+\mathcal{O}\big((-z)^{-\frac{3}{2}}\big)\big)\ee^{-\frac{\sqrt{2}}{3}n (-z)^{\frac{3}{2}}}.
\end{gather*}

However, applying Lemma~\ref{lemma_sym}, we can use $\tau_n(z)=\ee^{\frac{2\pi\ii}{3}n(n-1)}\tau_n\big(z\ee^{\frac{2\pi\ii}{3}}\big)$, and then the seed function in the region $|\arg(-z)|<\frac{\pi}{3}$ is $\varphi\big(z\ee^{\frac{2\pi\ii}{3}}\big)$, with constants
\begin{gather*}
C_1=\tfrac{1}{2}\ee^{-\frac{\pi\ii}{3}}, \qquad C_2=\tfrac{1}{2}\ee^{\frac{\pi\ii}{6}}.
\end{gather*}

After simplification, the first two coefficients $\mathbf{A}_{n,0}\big(z\ee^{\frac{2\pi\ii}{3}}\big)$ and $\mathbf{A}_{n,1}\big(z\ee^{\frac{2\pi\ii}{3}}\big)$ give
\begin{gather*}
\tau_n(z)=(-1)^{\left \lfloor{n/2}\right \rfloor} K_n G(n+1)(-z)^{-\frac{n^2}{4}}\ee^{-\frac{\sqrt{2}}{3}n(-z)^{\frac{3}{2}}}\\
\hphantom{\tau_n(z)=}{} \times \left[ 1 +(-1)^n \frac{2^{\frac{5n-7}{2}}\ii (-z)^{\frac{3}{2}(n-1)}}{\Gamma(n)}\ee^{\frac{2\sqrt{2}}{3}(-z)^{\frac{3}{2}}}
\right]\big(1+\mathcal{O}\big((-z)^{-\frac{3}{2}}\big)\big).
\end{gather*}

Other exponentially small contributions can be calculated in a similar way. As a consequence, we have a similar Stokes phenomenon for the Painlev\'e functions:
\begin{gather*}
\sigma_n(z)=-\frac{n(-z)^{1/2}}{\sqrt{2}}
\left[1+(-1)^{n+1}\frac{2^{\frac{5n-5}{2}}\ii}{\Gamma(n+1)}(-z)^{\frac{3}{2}(n-1)} \ee^{\frac{2\sqrt{2}}{3}(-z)^{\frac{3}{2}}}\right]\big(1+\mathcal{O}\big((-z)^{-\frac{3}{2}}\big)\big),
\\
p_n(z)=-\frac{n}{\sqrt{2} (-z)^{1/2}}
\left[1+(-1)^{n+1}\frac{2^{\frac{5n}{2}-2}\ii}{\Gamma(n)}(-z)^{\frac{3}{2}n}\ee^{\frac{2\sqrt{2}}{3}(-z)^{\frac{3}{2}}}\right]\big(1+\mathcal{O}\big((-z)^{-\frac{3}{2}}\big)\big),
\end{gather*}
and
\begin{gather*}
q_n(z)=\frac{(-z)^{1/2}}{\sqrt{2}}
\left[1+(-1)^{n+1}\frac{2^{\frac{5n-5}{2}}\ii}{\Gamma(n)}(-z)^{\frac{3}{2}(n-1)}\ee^{\frac{2\sqrt{2}}{3}(-z)^{\frac{3}{2}}}\right]\big(1+\mathcal{O}\big((-z)^{-\frac{3}{2}}\big)\big),
\end{gather*}
for $n\geq 1$.

Similar calculations can be carried out in the sector $\frac{4\pi}{3}<\arg z<2\pi$, using the seed function $\varphi\big(z\ee^{-\frac{2\pi\ii}{3}}\big)$ with constants
\begin{gather*} C_1=\tfrac{1}{2}\ee^{\frac{\pi\ii}{3}}, \qquad C_2=\tfrac{1}{2}\ee^{-\frac{\pi\ii}{6}}. \end{gather*}

As mentioned before, these extra exponential terms are not strictly needed in the non-oscillatory regime $|\arg(-z)|<\pi$, but they become relevant when one examines the transition between exponential and trigonometric behavior.

An alternative way to calculate these subleading exponential terms would be to extend the asymptotic expansion obtained in the oscillatory regime, with trigonometric functions, to a wider sector in $\mathbb{C}$ around the positive real axis. Then in the overlapping region between non-oscillatory and oscillatory behavior the two asymptotic expansions should coincide, and it would be possible to extract the subleading terms.

\subsection*{Acknowledgements}
The author acknowledges financial support from the EPSRC grant ``Painlev\'e equations: analytical properties and numerical computation", reference EP/P026532/1, and from the project MTM2015-65888-C4-2-P from the Spanish Ministry of Economy and Compe\-titi\-vity. The author wishes to thank M.~Fasondini, D.~Huybrechs, A.~Iserles, A.R.~Its, A.B.J.~Kuijlaars, A.F.~Loureiro, C.~Pech and W.~Van Assche for stimulating discussions on the topic and scope of this paper, as well as the organisers of the workshop ``Painlev\'e Equations and Applications'' held at the University of Michigan, August 25--29, 2017, for their hospitality. The comments, remarks and corrections of the anonymous referees have lead to an improved version of the paper, and they are greatly appreciated.

\pdfbookmark[1]{References}{ref}
\LastPageEnding

\end{document}